\documentclass[11pt]{article}
\usepackage{amssymb,amsmath,amsthm,mathdesign}
\usepackage{pdfsync}
\usepackage{color}
\usepackage[colorlinks]{hyperref}

\newtheorem{theorem}{Theorem}[section]
\newtheorem{lemma}[theorem]{Lemma}
\newtheorem{proposition}[theorem]{Proposition}
\newtheorem{corollary}[theorem]{Corollary}
\newtheorem{assumption}[theorem]{Assumption}
\newtheorem{remark}[theorem]{Remark}

\newtheorem{definition}[theorem]{Definition}
\numberwithin{equation}{section}

\newcommand{\be}{\begin{equation}}
\newcommand{\ee}{\end{equation}}

\newcommand{\bes}{\begin{equation*}}
\newcommand{\ees}{\end{equation*}}

\def\E{\bE}

\def\cG{\mathcal{G}}

\def\bE{\mathbb{E}}

\newcommand{\R}{\mathbb{R}}

\renewcommand{\d}{{\rm d}}

\renewcommand{\geq}{\geqslant}
\renewcommand{\leq}{\leqslant}
\renewcommand{\ge}{\geqslant}
\renewcommand{\le}{\leqslant}
\newcommand{\rd}{{\mathbb R^d}}

\def\m1{\mathbf{1}}

\allowdisplaybreaks[1]

\allowdisplaybreaks[1]

\title{Critical parameters for reaction-diffusion equations involving  space-time fractional  derivatives.}
\author{Sunday A. Asogwa\\ Auburn University
\and Mohammud Foondun\\University of Strathclyde
\and Jebessa B. Mijena\\Georgia College \& State University
\and Erkan Nane\\ Auburn University
 }
\date{}
\begin{document}
\maketitle

\begin{abstract}
We will look at reaction-diffusion type equations of the following type,
$$\partial^\beta_tV(t,x)=-(-\Delta)^{\alpha/2} V(t,x)+I^{1-\beta}_t[V(t,x)^{1+\eta}].$$
We first study the equation on the whole space by making sense of it via an integral equation. Roughly speaking, we will show that when $0<\eta\leq\eta_c$, there is no global solution other than the trivial one while for $\eta>\eta_c$, non-trivial global solutions do exist. The critical parameter $\eta_c$ is shown to be $\frac{1}{\eta^\ast}$ where
\begin{equation*}
\eta^\ast:=\sup_{a>0}\left\{\sup_{t\in (0,\,\infty),x\in \R^d}t^a\int_{\R^d}G(t,\,x-y)V_0(y)\,\d y<\infty\right\}
\end{equation*}
and $G(t,\,x)$ is the heat kernel of the corresponding unforced operator.  $V_0$ is a non-negative initial function.  We also study the equation on a bounded domain with Dirichlet boundary condition and show that the presence of the time derivative induces a significant change in the behaviour of the solution.

\end{abstract}

\section{Introduction and main results}
A very influential paper by Fujita \cite{Fujita} looks at the following equation
\begin{align}\label{usual}
 \partial_tu(t, x)&=
 \Delta u(t, x)+u(t,\,x)^{1+\eta}\quad x\in \R^d\\
 u(0, x)&=u_0(x).\nonumber
\end{align}

Let $\eta_c=\frac{2}{d}$.  It was shown in \cite{Fujita} that when $0<\eta<\eta_c$, there is no nontrivial global solution no matter how  small the initial condition $u_0$ is, provided it is nonnegative. When $\eta>\eta_c$, then one can construct nontrivial global solution when $u_0$ is small enough.  The critical case $\eta=\eta_c$ was shown to fall into the first category; see \cite{Haya} and \cite{KST}. These results have inspired a lot of generalisations. See the survey papers \cite{Levine} and \cite{DL} and the book \cite{Souplet}. Eq\eqref{usual} could be interpreted via the integral equation
\begin{equation}\label{fujita}
u(t,\,x)=\int_{\R^d}p(t,\,x-y)u_0(y)\,\d y+\int_0^t\int_{\R^d}p(t-s,\,x-y)u(s,\,y)^{1+\eta}\,\d y\,\d s,
\end{equation}
where $p(t,\,x)$ is the Gaussian heat kernel.  This is the approach we adopt here.

Roughly speaking, our aim here is to look at similar questions but for a class of equations which involve the fractional Laplacian as well as a fractional time derivative. Equations of these types have been receiving a lot of attention lately; see the recent works of Allen, Caffarelli and Vasseur; \cite{ACV} and \cite{ACV2} and of Allen; \cite{Allen1} and \cite{Allen2} among others on the purely analytic side and the very recent work of Capitanelli and D'Ovidio \cite{CD1} and references therein for the more probabilistic aspects.  Consider the following generalisation of \eqref{fujita},
\begin{equation}\label{mild}
V(t, x)=
\int_{\R^d} G(t,\,x-y)V_0(y)\,\d y + \int_{\R^d}\int_0^t G(t-s,\,x-y)V(s, y)^{1+\eta}\d s\,\d y.
\end{equation}
The first term in the above display now solves the  space-time fractional  heat equation
\begin{align}\label{H1}
 \partial^\beta_tV(t, x)&=
 -(-\Delta)^{\alpha/2} V(t, x)\ \ x\in \mathbb{R}^d\\
 V(0, x)&=V_0(x), \nonumber
\end{align}
where $\alpha\in (0,\,2)$ and $\beta\in (0,\,1)$. The fractional time derivative is the Caputo derivative defined by
\begin{equation*}\label{CaputoDef}
\partial^\beta_t V(t,x)=\frac{1}{\Gamma(1-\beta)}\int_0^t \frac{\partial
V(r,x)}{\partial r}\frac{\d r}{(t-r)^\beta}.
\end{equation*}
 The solution to \eqref{mild} is referred to as the integral solution to the following equation
\begin{equation}\label{eq:additional source}
{
\begin{split}
 \partial^\beta_tV(t, x)&=
 -(-\Delta)^{\alpha/2} V(t, x)+I^{1-\beta}_t[ V(t, x)^{1+\eta}]\\
 V(0, x)&=V_0(x).
 \end{split}
 }
 \end{equation}
The operator $-(-\Delta)^{\alpha/2} $ denotes the fractional Laplacian which is the generator of an $\alpha$-stable process. $V_0$ will always be assumed to be a non-negative function.  We will impose further assumptions on $V_0$ later.  The operator $I^{1-\beta}_t$ is defined by
\begin{equation*}
I^{1-\beta}_{t}f(t):=\frac{1}{\Gamma(1-\beta)}\int_{0}^{t}(t-\tau)^{-\beta}f(\tau)\d\tau,
\end{equation*}
and will not play any role in this paper. Its presence is however important in making the connection between \eqref{eq:additional source} and \eqref{mild}.  See \cite{umarov-12} for the fractional Duhamel's principle.
A  version of  equation  \eqref{eq:additional source} is called time-fractional Bloch--Torrey equation for fractional diffusion tensor imaging in \cite{meerschaert-et-al-2016}. See  section 6.2 in their paper for more details. 

 Our main findings can be summarised as follows:

\begin{itemize}
\item We show that $\eta_c=\frac{\alpha}{\beta}$. This is a direct generalisation of the dichotomy first discovered in \cite{Fujita}, \cite{Haya} and \cite{KST}. When $\beta=1$ and $\alpha=2$, \eqref{mild} becomes \eqref{fujita}. Our new found exponent is therefore consistent with that obtained in \cite{Fujita}.
\item We also study \eqref{mild} on a bounded domain with Dirichlet boundary conditions.  For the usual heat equation, that is with the usual time derivative and Laplacian, there is no such dichotomy. This means that one can always produce global solutions no matter what $\eta$ is ; See \cite{Souplet}. In our case, we show that this is not true; for small $\eta$, there is no global solution other than the trivial one.
\end{itemize}

We focus only on {\it integral solution} to \eqref{eq:additional source} as defined say on page 78 in the  book \cite{Souplet} which also contains a list of other concepts of solution. There are also various meanings of non-existence or blow-up of solution, we will be focusing mainly   {on} pointwise non-existence. See \cite{Levine} or \cite{Souplet}, where this is explained in great details. Our method will rely on some new estimates on the heat kernel associated with \eqref{H1} some of which were first proved in \cite{FooNan} and later extended in \cite{KKW}. We will make use of {\it subordination} to get new information about the heat kernel. See \eqref{Eq:Green2} of this current paper.  A difficulty in establishing non-existence on the whole line is that the heat kernel does not satisfy the semigroup property. We had to establish a new strategy to achieve our first result.  Since we had to bypass the semigroup property our method might even be new in the classical heat equation; that is ;  when $\alpha=2$ and $\beta=1$. Our first theorem reads as follows.

\begin{theorem}\label{thm-additional-drift}
Suppose that $0<\eta\leq \alpha/\beta d$ and $V_0\not\equiv 0$, then there is no global solution to \eqref{eq:additional source} in the sense that there exists a $t_0>0$ such that $V(t,\,x)=\infty$ for all $t>t_0$ and $x\in \R^d$.
\end{theorem}
The above theorem generalises Theorem 18.3 of \cite{Souplet} but the method is different.  The presence of the time fractional derivative makes it that when $\alpha\leq d$, the heat kernel has a singularity at $x=0$ for all $t>0$. This partly motivated the proof of the next theorem. 
\begin{theorem}\label{thm-norm-existence}
Suppose that $\eta> \alpha/\beta d$. Then, there are initial conditions $V_0$ for which the solutions to \eqref{eq:additional source} exist globally. 
\end{theorem}
In fact the above result will be a consequence of another result which says that for some $p>1$,  $\|V(t,\,\cdot)\|_{L^p(\R^d)}$ decays polynomially. This is also  an extension over previously known results.  We will also show that the solution is jointly continuous whenever it exists. Even though regularity properties of the solution is not a priority here, our results in this direction seems to be new. When $d<\alpha$, we have better estimates on the heat kernel so that we can establish the following stronger result.  Since $\alpha\in(0,\,2)$, this condition reduces the dimension to $d=1$. The theorem below significantly extends Theorem 20.1 of \cite{Souplet}.

 \begin{theorem}\label{thm-decay}
Let $d<\alpha$ and $\eta>\alpha/\beta d$. Suppose that for some small $\delta>0$,  $V_0$ satisfies
\begin{equation*}
0\leq V_0(x)\leq \delta G(\gamma,\,x)\quad \text{for\,all}\quad x\in \R^d,
\end{equation*}
where $\gamma$ is a positive constant. We then have,
\begin{equation*}
V(t,\,x)\lesssim G(t+\gamma,\,x).
\end{equation*}
Moreover, the solution is jointly continuous on $(0,\,\infty)\times \R^d$.
\end{theorem}
We have therefore shown that $\eta_c=\frac{\alpha}{\beta d}$. This is consistent with the following characterisation which says that this exponent is the reciprocal of the following quantity,
\begin{equation*}
\eta^\ast:=\sup_{a>0}\left\{\sup_{t\in (0,\,\infty),x\in \R^d}t^a\int_{\R^d}G(t,\,x-y)V_0(y)\,\d y<\infty\right\}.
\end{equation*}
Indeed one can show that the supremum of $\int_{\R^d}G(t,\,x-y)V_0(y)\,\d y$ behaves like $t^{-\beta d/\alpha}$.  This characterisation also gives $\eta_c=0$ when \eqref{usual} is solved on a bounded domain with Dirichlet boundary condition. See page 108 of \cite{Souplet} where this is described in more details.  Our next result shows that this is not true when one looks at the corresponding equation with a time-fractional derivative.  Fix $R>0$ and consider the following
\begin{equation}\label{eq:dir-additional-source}
\begin{split}
 \partial^\beta_tV(t, x)&=
 -(-\Delta)^{\alpha/2} V(t,x)+I^{1-\beta}_t[ V(t, x)^{1+\eta}]\quad{t>0}\quad\text{and}\quad x\in B(0,\,R),\\
 V(t, x)&= 0\ \ x\in B(0,\,R)^c\\
 V(0, x)&=V_0(x) \ \ x\in B(0,\,R).
 \end{split}
 \end{equation}
Here $-(-\Delta)^{\alpha/2}$ denotes the generator of  $\alpha$-stable process killed upon exiting the ball $B(0,\,R)$.  We will again look at the integral formulation of the equation,
\begin{equation}\label{mild2}
V(t, x)=
\int_{B(0,\,R)} G_D(t,\,x,\,y)V_0(y)\,\d y + \int_{B(0,\,R)}\int_0^t G_D(t-s,\,x,\,y)V(s, y)^{1+\eta}\d s\,\d y,
\end{equation}
where now $G_D(t,\,x,\,y)$ is the Dirichlet heat kernel of the associated operator. Denote $\phi_1$ to be  the first  eigenfunction of the above Dirichlet fractional Laplacian and set
\begin{equation*}
K_{V_0, \phi_1}:=\int_{B(0,\,R)}V_0(x)\phi_1(x)\,\d x.
\end{equation*}
We are now ready to state the final theorem of this paper.  This is a consequence of the spectral decomposition of the heat kernel in terms of Mittag-Leffler functions and the proof uses the {\it eigenfunction method} of \cite{Kaplan}. The first part of this theorem is in sharp contrast with Theorem 19.2 of \cite{Souplet}.
\begin{theorem}\label{thm-Dirichlet-additional-drift}
Suppose that $0<\eta<1/\beta-1$, then there is no global solution to \eqref{eq:dir-additional-source} whenever $K_{V_0, \phi_1}>0$. For any $\eta>0$,  there is no global solution whenever $K_{V_0, \phi_1}>0$ is large enough.
\end{theorem}
At this point we do not investigate the dichotomy as in the equation on the whole plane.  One can perhaps argue that since the solution to the Dirichlet equation is smaller than that on the whole plane, we can find global solution when $\eta$ is large enough.

Here is a plan of the article. Section \ref{sect-prelim} contains estimates needed for the proof of Theorem \ref{thm-additional-drift}. This is given in Section \ref{proof2}. Section \ref{proof3} is devoted to the proof of Theorem \ref{thm-norm-existence} while the proofs of Theorem \ref{thm-decay} and Theorem \ref{thm-Dirichlet-additional-drift} are given in Section \ref{proof4} and Section \ref{proof5} respectively.  We use the notation $f(t,x)\lesssim (\gtrsim)g(t,x)$ when there exists a constant $C$ independent of $(t,x)$ such that $f(t,x)\leq (\geq) C g(t,x)$ for all $(t, x)\in (0,\,\infty)\times \R^d$.

\section{Some estimates}\label{sect-prelim}
We begin this section by giving a brief description of the process associated with \eqref{H1}. { However; we will not use this process directly.  Instead we will use it } to derive a suitable representation of its heat kernel. See \cite{CME} and \cite{mnv-09}  for more information.
Let $X_t$ denote a symmetric $\alpha$ stable process associated with the fractional Laplacian. Its density function will be denoted by $p(t,\,x)$. This is characterized through the Fourier transform which is given by

\begin{equation*}
\widehat{p(t,\,\xi)}=\int_\rd e^{\xi\cdot x}p(t,x)\d x=e^{-t|\xi|^\alpha}.
\end{equation*}
The following properties of $p(t,x)$ will be needed in this paper:
 \begin{itemize}
\item \begin{equation*}\label{stable-kernel-scaling}
p(st, x)=s^{-d/\alpha}p(t,s^{-1/\alpha}x).
\end{equation*}
\item
\begin{equation}\label{t-deri}
\frac{\partial p(t,\,x)}{\partial t}\lesssim \frac{1}{t}p(t,\,x).
\end{equation}
\item
\begin{equation}\label{s-deri}
\nabla p(t,\,x)\lesssim \frac{1}{t^{1/\alpha}}p(t,\,x).
\end{equation}
\item For all $t>0$, $x,y \in \R^d$ and $\rho\in[0,\,1]$,
\begin{equation}\label{x-diff}
{|p(t,\,y)-p(t,\,x)|\lesssim \frac{|x-y|^\rho}{t^{\rho/\alpha}}[p(t,\,x/2)+p(t,\,y/2)].}
\end{equation}

\end{itemize}
We also have
\begin{equation}\label{heatalpha}
  c_1\bigg(t^{-d/\alpha}\wedge \frac{t}{|x|^{d+\alpha}}\bigg)\leq p(t,x)\leq c_2\bigg(t^{-d/\alpha}\wedge \frac{t}{|x|^{d+\alpha}}\bigg),
 \end{equation}
 for some positive constants $c_1$ and $c_2$; see for instance \cite{kolo}.  The process associated with \eqref{H1} is not Markov and the heat kernel $G(t,\,x)$ does not satisfy the semigroup property.  We describe this process next. Let $D=\{D_r,\,r\ge0\}$ be a $\beta$-stable subordinator with $\beta\in (0,1)$. Its Laplace transform is given by $\E(e^{-sD_t})=e^{-ts^\beta}$. Let $E_t$ be its first passage time. The process which we will be interested in is given { by} the time changed process $X_{E_t}$. This is the process associated with the time fractional heat equation given by \eqref{H1}.  Its density $G(t,\,x)$ is given by a simple  conditioning as follows

\begin{equation}\label{Eq:Green1}
G(t,\,x) = \int_{0}^\infty p(s,\,x) f_{E_t}(s)\d s,
\end{equation}
where
\begin{equation*}\label{Etdens0}
f_{E_t}(x)=t\beta^{-1}x^{-1-1/\beta}g_\beta(tx^{-1/\beta}).
\end{equation*}
The function $g_\beta(\cdot)$ is the density function of $D_1$ and is infinitely differentiable on the entire real line, with $g_\beta(u)=0$ for $u\le 0$. After a change of variable, \eqref{Eq:Green1} turns into
\begin{equation}\label{Eq:Green2}
G(t,\,x)=\int_0^\infty p\left(\left(\frac{t}{u}\right)^\beta, x\right)g_\beta(u)\,\d u,
\end{equation}
which makes the following asymptotic properties particularly useful,
\begin{equation}\label{Eq:gbeta0}
g_\beta(u)\sim K(\beta/u)^{(1-\beta/2)/(1-\beta)}\exp\{-|1-\beta|(u/\beta)^{\beta/(\beta-1)}\}\quad\mbox{as}\,\, u\to0+,
\end{equation}
and
\begin{equation}\label{Eq:gbetainf}
g_\beta(u)\sim\frac{\beta}{\Gamma(1-\beta)}u^{-\beta-1} \quad\mbox{as}\,\, u\to\infty.
\end{equation}
Using \eqref{Eq:Green2} together with \eqref{stable-kernel-scaling}, we obtain
\begin{itemize}
\item \begin{equation}\label{time-kernel-scaling}
G(st, x)=s^{-\beta d/\alpha}G(t,s^{-\beta/\alpha}x),
\end{equation}
\end{itemize}
As explained above, our method will be partly inspired by the following inequality which was first proved in \cite{FooNan} and subsequently generalised in \cite{KKW}.

 \begin{equation}\label{heat}
 c_1\bigg(t^{-\beta d/\alpha}\wedge \frac{t^\beta}{|x|^{d+\alpha}}\bigg)\leq G(t,\,x)\leq c_2\bigg(t^{-\beta d/\alpha}\wedge \frac{t^\beta}{|x|^{d+\alpha}}\bigg),
 \end{equation}
where the upper bound is valid for $\alpha > d$ only. In this case, we immediately have
\begin{equation}\label{heat-comp}
p(t^\beta,\,x)\lesssim G(t,\,x)\lesssim p(t^\beta,\,x),
\end{equation}
which we will use to compensate for the lack of the semigroup property. If $|x|\leq t^{\beta/\alpha}$, then when $\alpha=d$, we have
\begin{equation*}
t^{-\beta}\log\left( \frac{2}{|x|t^{-\beta/\alpha}}\right)\lesssim G(t,\,x)\lesssim t^{-\beta}\log\left( \frac{2}{|x|t^{-\beta/\alpha}}\right)
\end{equation*}
and when $d>\alpha$,
\begin{equation*}
\frac{t^{-\beta}}{|x|^{d-\alpha}}\lesssim G(t,\,x)\lesssim \frac{t^{-\beta}}{|x|^{d-\alpha}}.
\end{equation*}
When $|x|\geq t^{\beta/\alpha}$, then $G(t,\,x)$ satisfy the bounds given by \eqref{heat} even $d\geq \alpha$. This was shown in \cite{KKW}.  We have the following estimates on the derivatives of the heat kernel.

\begin{proposition}\label{der-heat}
For any $t>0$ and $x\in \R^d$, we have
\begin{itemize}
\item[(a)]\begin{equation}\label{time-deri}
\frac{\partial G(t,\,x)}{\partial t}\lesssim \frac{1}{t}G(t,\,x).
\end{equation}
\item[(b)]\begin{equation}\label{spatial-deri}
\nabla G(t,\,x)\lesssim \frac{1}{t^{\beta/\alpha}}G(t,\,x),
\end{equation}
whenever $\alpha>1$.

\item[(c)] Let $\rho<\alpha$, then we have
\begin{equation}\label{space-incre}
\int_{\R^d}|G(t,\,x+h)-G(t,\,x)|f(t,\,x)\,\d x\lesssim \frac{|h|^{\rho}}{t^{\rho\beta/\alpha}},
\end{equation}
where $h\in \R^d$ and $f(t,\,x)$ is a bounded function for each $t>0$.
\end{itemize}
\end{proposition}

\begin{proof}
The proofs of the first two parts follow from
\begin{equation*}
G(t,\,x)=\int_0^\infty p\left(\left(\frac{t}{u}\right)^\beta, x\right)g_\beta(u)\,\d u,
\end{equation*}
and \eqref{t-deri}, \eqref{s-deri} and the assymptotic properties of $g_\beta(u)$.
For the last part, we use \eqref{x-diff} to obtain
\begin{align*}
G(t,\,x+h)-G(t,\,x)&=\int_0^\infty \left[p\left(\left(\frac{t}{u}\right)^\beta, x+h\right)-p\left(\left(\frac{t}{u}\right)^\beta, x\right)\right]g_\beta(u)\,\d u\\
&\lesssim \frac{|h|^{\rho}}{t^{\rho\beta/\alpha}}\int_0^\infty u^{\rho\beta/\alpha}\left[p\left(\left(\frac{t}{u}\right)^\beta, \frac{x+h}{2}\right)+p\left(\left(\frac{t}{u}\right)^\beta, \frac{x}{2}\right)\right]g_\beta(u)\,\d u.
\end{align*}
Hence, we have
\begin{align*}
\int_{\R^d}|G(t,\,x+h)-G(t,\,x)|f(t,\,x)\,\d x&\lesssim\frac{|h|^{\rho}}{t^{\rho\beta/\alpha}}\int_0^\infty u^{\rho\beta/\alpha}g_\beta(u)\,\d u \\
&\lesssim \frac{|h|^{\rho}}{t^{\rho\beta/\alpha}}.
\end{align*}
That the integral appearing on the right hand side of above display is finite when $\rho<\alpha$ can be seen by looking at the behaviour of $g_\beta(u)$ as $u\rightarrow \infty.$
\end{proof}
Set  \begin{equation*}
\mathcal{G}f(t,\,x):=\int_{\R^d}G(t,\,x-y)f(y)\,\d y,
\end{equation*}
and
\begin{equation*}
\mathcal{A}f(t,\,x):=\int_0^t\int_{\R^d}G(t-s,\,x-y)f(s,y)^{1+\eta}\,\d y\,\d s.
\end{equation*}
We will need the following to argue that the solution is jointly continuous whenever it exists.
\begin{proposition}\label{cty}
\begin{itemize}
\item Suppose that $V_0$ is such that $\sup_{(0,\,T)\times \R^d} \mathcal{G}V_0(t,\,x)<\infty$ for some $T\leq \infty$, then $\mathcal{G}V_0(t,\,x)$ is jointly continuous on $(0,\,T)\times \R^d$.

\item Suppose that $\sup_{t\in(0,T], x\in\R^d}f(t,\,x)<\infty$ for some $T\leq \infty$. Then $\mathcal{A}f(t,\,x)$ is jointly continuous on $(0,\,T)\times \R^d$.
\end{itemize}
\end{proposition}

\begin{proof}
The proof uses Proposition \ref{der-heat}. We merely indicate the how to start the proof of the more technical part. For $h>0$, $k\in \R^d$, we write
\begin{align*}
\mathcal{A}f(t+h,\,x+k)-\mathcal{A}f(t,\,x)&=\mathcal{A}f(t+h,\,x+k)-\mathcal{A}f(t,\,x+k)\\
&+\mathcal{A}f(t,\,x+k)-\mathcal{A}f(t,\,x)\\
&:=I+II.
\end{align*}
For the first part, we have
\begin{align*}
I&=\int_0^{t+h}\int_{\R^d}G(t+h-s,\,x+k-y)f(s,y)^{1+\eta}\,\d y\,\d s-\int_0^t\int_{\R^d}G(t-s,\,x+k-y)f(s,y)^{1+\eta}\,\d y\,\d s\\
&=\int_0^t\int_{\R^d}[G(t+h-s,\,x+k-y)-G(t-s,\,x+k-y)]f(s,y)^{1+\eta}\,\d y\,\d s\\
&+\int_t^{t+h}\int_{\R^d}G(t+h-s,\,x+k-y)f(s,y)^{1+\eta}\,\d y\,\d s.
\end{align*}
We can now use the above Proposition to bound each term. We deal with the second part in a similar fashion.

\end{proof}

\begin{lemma}\label{lemma:lower-bound-heq}
There exists a $T>0$, such that for all $t\geq T$,
\begin{align*}
\mathcal{G}V_0(t,\,x)\gtrsim \frac{1}{t^{\beta d/\alpha}}\quad\text{for all}\quad x\in B(0,\,t^{\beta/\alpha}).
\end{align*}
\end{lemma}
\begin{proof}
Let $x\in B(0,\,t^{\beta/\alpha})$. We now use the lower bound on the heat kernel to write
\begin{align*}
\mathcal{G}V_0(t,\,x)&= \int_{\R^d}G(t,\,x-y)V_0(y)\,\d y\\
&\geq \int_{B(0,\,t^{\beta/\alpha})}G(t,\,x-y)V_0(y)\,\d y\\
&\gtrsim \frac{1}{t^{\beta d/\alpha}} \int_{B(0,\,t^{\beta/\alpha})}V_0(y)\,\d y.
\end{align*}
By choosing $t$ large enough, we obtain the desired inequality.
\end{proof}

\section{Proof of Theorem \ref{thm-additional-drift}}\label{proof2}

\begin{proposition}\label{apriori-1}
Suppose that $\eta\leq\frac{\alpha}{\beta d}$. Let $M>0$, then there exists a $T_0>0$ such that for $t\geq T_0$,
\begin{equation*}
\inf_{x\in B(0,\,t^{\beta/\alpha})}V(t,\,x)\geq M.
\end{equation*}
\end{proposition}
\begin{proof}
We begin with the integral solution,
\begin{align*}
V(t, x)&=
\mathcal{G}V_0(t,\,x)+ \int_{\R^d}\int_0^t G(t-s,\,x-y)V(s, y)^{1+\eta}\d s\,\d y.
\end{align*}
We look at the second term first. For $x\in B(0,\,t^{\beta/\alpha})$, we have
\begin{align*}
\int_{\R^d}\int_0^t &G(t-s,\,x-y)V(s, y)^{1+\eta}\d s\,\d y\\
&\geq \int_0^t\inf_{y\in B(0,\,s^{\beta/\alpha})}V(s,\,y)^{1+\eta}\int_{B(0,\,s^{\beta/\alpha})}G(t-s,\,x-y)\d y\,\d s\\
&\geq \int_0^{t/2}\inf_{y\in B(0,\,s^{\beta/\alpha})}V(s,\,y)^{1+\eta}\int_{B(0,\,s^{\beta/\alpha})}G(t-s,\,x-y)\d y\,\d s\\
&\gtrsim \int_0^{t/2}\inf_{y\in B(0,\,s^{\beta/\alpha})}V(s,\,y)^{1+\eta}\frac{s^{\beta d/\alpha}}{t^{\beta d/\alpha}}\,\d s,
\end{align*}
where we have used the lower bounds given by \eqref{heat}. For the first term we use Lemma \ref{lemma:lower-bound-heq} to write
\begin{align*}
\inf_{x\in B(0,\,t^{\beta/\alpha})}\mathcal{G}V_0(t,\,x)\gtrsim \frac{1}{t^{\beta d/\alpha}},
\end{align*}
whenever $t$ is large enough. Combining these estimates, we obtain
\begin{align*}
\inf_{x\in B(0,\,t^{\beta/\alpha})}V(t,\,x)\gtrsim  \frac{1}{t^{\beta d/\alpha}}+\int_0^{t/2}\inf_{y\in B(0,\,s^{\beta/\alpha})}V(s,\,y)^{1+\eta}\frac{s^{\beta d/\alpha}}{t^{\beta d/\alpha}}\,\d s.
\end{align*}
Set
\begin{align*}
F(t):=\inf_{x\in B(0,\,t^{\beta/\alpha})}t^{\beta d/\alpha}V(t,\,x),
\end{align*}
If $\eta<\frac{\alpha}{\beta d}$, the above inequality reduces to
\begin{align*}
F(t)\gtrsim 1+\int_0^{t/2}\frac{F(s)^{1+\eta}}{s^{\eta \beta d/\alpha}}\,\d s.
\end{align*}
Some computations imply that for any given fixed integer $N>0$, there are strictly positive constants $c_N$ and $\tilde{c}_N$ such that 
\begin{align*}
F(t)\gtrsim \tilde{c}_Nt^{c_N}.
\end{align*}
By taking $t$ large enough, we obtain more than what we need. When $\eta=\frac{\alpha}{\beta d}$, we obtain
\begin{align*}
F(t)\gtrsim 1+\int_1^{t/2}\frac{F(s)^{1+\eta}}{s}\,\d s,
\end{align*}
which again gives us what we need.
\end{proof}
A consequence of the above is the following.

\begin{proposition}\label{apriori-2}
Let $\eta\leq\frac{\alpha}{\beta d}$, then for $T$ large enough
\begin{align*}
\int_0^T\int_{\R^d}V(s,\,y)^{1+\eta}G(T+t-s,\,x-y)\,\d s\,\d y\gtrsim T,
\end{align*}
whenever $0<t<\frac{T}{3}$ and $x\in B(0,\,T^{\beta/\alpha})$.
\end{proposition}
\begin{proof}
We use the previous proposition to write
\begin{align*}
\int_0^T\int_{\R^d}V(s,\,y)^{1+\eta}&G(T+t-s,\,x-y)\,\d y\,\d s\\
&\geq \int_{(T+t)/2}^{3(T+t)/4}\int_{B(0,\,s^{\beta/\alpha})}V(s,\,y)^{1+\eta}G(T+t-s,\,x-y)\,\d y\,\d s\\
&\geq M^{1+\eta} \int_{(T+t)/2}^{3(T+t)/4}\int_{B(0,\,s^{\beta/\alpha})}G(T+t-s,\,x-y)\,\d y\,\d s.
\end{align*}
Since $t<\frac{T}{3}$, we have $B(0,\,(T+t-s)^{\beta/\alpha})\subset B(0,\,s^{\beta/\alpha})$ and $|x-y|\leq c_1(T+t-s)^{\beta/\alpha}$. We therefore have
\begin{align*}
\int_{B(0,\,s^{\beta/\alpha})}&G(T+t-s,\,x-y)\,\d y\\
&\geq \int_{B(0,\,(T+t-s)^{\beta/\alpha})}G(T+t-s,\,x-y)\,\d y\\
&\gtrsim 1,
\end{align*}
where we have used the lower bound given by \eqref{heat} to obtain the last inequality.  We combine these estimates above to obtain the result.
\end{proof}


We are now ready to prove  Theorem \ref{thm-additional-drift}.
\begin{proof}[Proof of Theorem \ref{thm-additional-drift}]
Let $T>0$ which we are going to fix later. From the intergral solution, we have
\begin{equation*}
V(t+T, x)=
\int_{\R^d} G(t+T,\,x-y)V_0(y)\,\d y + \int_{\R^d}\int_0^{t+T} G(t+T-s,\,x-y)V(s, y)^{1+\eta}\\d s\,\d y.
\end{equation*}
A simple change of variables and the fact that the first term of the above display is non-negative, we obtain
\begin{align*}
V(t+T, x)&\geq \int_{\R^d}\int_0^{T} G(t+T-s,\,x-y)V(s, y)^{1+\eta}\d s\,\d y\\
&+\int_{\R^d}\int_0^{t} G(t-s,\,x-y)V(s+T, y)^{1+\eta}\d s\,\d y.
\end{align*}
We bound the first term of the above display.  From the above proposition, for $x\in B(0,\,1)$, we have upon taking $T$ large enough,
\begin{align*}
\int_{\R^d}\int_0^{T} G(t+T-s,\,x-y)V(s, y)^{1+\eta}\d s\,\d y\gtrsim T.
\end{align*}
We now look at the second term. We take $t\leq \left(\frac{1}{2}\right)^{\alpha/\beta}$.
\begin{align*}
\int_{\R^d}\int_0^{t} &G(t-s,\,x-y)V(s+T, y)^{1+\eta}\d s\,\d y\\
&\geq \int_0^{t} \inf_{y\in B(0,\,1)}V(s+T,\,y)^{1+\eta}\int_{B(0,1)}G(t-s,\,x-y)\d y\,\d s.
\end{align*}
Since $x\in B(0,\,1)$ we can take $t\leq \left(\frac{1}{2}\right)^{\alpha/\beta}$ and use the heat kernel estimates to obtain
\begin{align*}
\int_{B(0,1)}&G(t-s,\,x-y)\d y\\
&\geq \int_{B(0,(t-s)^{\beta/\alpha})\cap B(0,\,1)}G(t-s,\,x-y)\d y\\
&\gtrsim 1.
\end{align*}
Putting these estimates together yield
\begin{align*}
\inf_{x\in B(0,\,1)}V(t+T,\,x)\gtrsim T+\int_0^t \inf_{x\in B(0,\,1)}V(s+T,\,x)^{1+\eta}\,\d s.
\end{align*}
Fix $T$ large enough so that $\inf_{x\in B(0,\,1)}V(t+T,\,x)=\infty$ for all $t\in [t_0, \left(\frac{1}{2}\right)^{\alpha/\beta}]$.  We now use the integral solution again to conclude that there exists a $T_0>0$ such that for all $t\geq T_0$, $V(t,\,x)=\infty$. 
\end{proof}

\section{Proof of Theorem \ref{thm-norm-existence}}\label{proof3}
The proof of the following result is a straightforward application of Young's convolution inequality.
\begin{lemma}\label{young}
For all $t>0$, we have
\begin{enumerate}
\item[(a)]\begin{equation*}
\|\mathcal{G}V_0(t,\,\cdot) \|_{L^r(\R^d)}\lesssim t^{-\frac{\beta d}{\alpha}(\frac{1}{p}-\frac{1}{r})}\|V_0\|_{L^p(\R^d)}
\end{equation*}
with $p,\,r\in [1,\,\infty]$ satisfying $0\leq \frac{1}{p}-\frac{1}{r}<\frac{\alpha}{d}$
\item[(b)] For $0\leq s\leq t$, we have
\begin{equation}\label{B1}
\| \int_{\R^d}G(t-s,\,\cdot-y)f(s,y)^{1+\eta}\,\d y\|_{L^r(\R^d)} \lesssim (t-s)^{-\frac{\beta d}{\alpha}(\frac{1+\eta}{p}-\frac{1}{r})}\|f(s,\,\cdot)\|_{L^p(\R^d)}^{1+\eta}
\end{equation}
with $\frac{p}{1+\eta}, r\in [1,\,\infty]$ satisfying $0\leq\frac{1+\eta}{p}-\frac{1}{r}<\frac{\alpha}{d}.$
\end{enumerate}
\end{lemma}
\begin{proof}
Young's convolution inequality gives us
\begin{equation*}
\|\mathcal{G}V_0(t,\,\cdot) \|_{L^r(\R^d)}\leq \|G(t,\,\cdot)\|_{L^q(\R^d)}\|V_0\|_{L^p(\R^d)},
\end{equation*}
for any $p, q, r\in [1,\infty]$ satisfying $1+\frac{1}{r}=\frac{1}{p}+\frac{1}{q}.$
The first part now follows by noting that from the scaling property and the heat kernel estimates,
\begin{align*}
\|G(t,\,\cdot)\|_{L^q(\R^d)}\lesssim t^{-\frac{\beta d}{\alpha}(1-\frac{1}{q})},
\end{align*}
whenever $1-\frac{1}{q}<\frac{\alpha}{d} $.
For the second inequality, we use Young's inequality again and the above but this time with parameters $\frac{p}{1+\eta}, q, r\in [1,\,\infty]$ satisfying $1+\frac{1}{r}=\frac{1+\eta}{p}+\frac{1}{q}$.
\end{proof}
For the next result, we will need the following notation.  
Set\begin{equation}\label{norm}
\|V\|_{p, \theta}:= \sup_{t>0}t^{\theta}\|V(t,\,\cdot)\|_{L^p(\R^d)}.
\end{equation}

\begin{corollary}\label{contract}
Suppose that $\eta>\frac{\alpha}{\beta d}$ and let $p> \frac{\beta d \eta}{\alpha}$.  Let $$\theta:=\frac{\beta d}{\alpha}\left(\frac{\alpha}{\beta d \eta}-\frac{1}{p}\right).$$ Then, we have
\begin{itemize}
\item[(a)] \begin{equation*}
\|\mathcal{G}f \|_{p, \theta}\lesssim \|f\|_{L^{q_c}(\R^d)},
\end{equation*}
where $q_c:=\frac{\beta d \eta}{\alpha}$ and $\theta/\beta<1$.
\item[(b)]\begin{equation*}
\| \mathcal{A}f\|_{p,\,\theta}\lesssim \|f\|^{1+\eta}_{p,\,\theta},
\end{equation*}
with $\frac{p}{1+\eta}\in [1,\,\infty]$ and $p>\frac{d\eta}{\alpha}$.
\item[(c)]Suppose that $f$ and $g$ satisfy $\|f\|_{p,\,\theta}<M$ and $\|g\|_{p,\,\theta}<M$ for some $M>0$. We then have
\begin{equation*}
\| \mathcal{A}f-\mathcal{A}g\|_{p,\,\theta}\lesssim M^\eta \|f-g\|_{p,\,\theta},
\end{equation*}
whenever $(1+\eta)\theta<1$, $\frac{p}{1+\eta}\in [1,\,\infty]$ and $p>\frac{d\eta}{\alpha}$.
\end{itemize}
\end{corollary}
\begin{proof} The first part is a straightforward consequence of the first part of the above Lemma \ref{young}. For the second part, the same lemma gives us
\begin{align*}
\|\mathcal{A}f\|_{L^p(\R^d)}\lesssim t^{1-\frac{\beta d\eta}{\alpha p}}\| f\|^{1+\eta}_{L^p(\R^d)},
\end{align*}
from which we obtain the result after some computations. The final part is slightly more involved. For the second inequality below, we use Young's inequality with parameters $1+\frac{1}{p}=\frac{\eta+1}{p}+\frac{p-\eta}{p}$ along with the assumption that $p>\frac{d\eta}{\alpha}$,
\begin{align*}
 \|\mathcal{A}f(t,\,\cdot)&-\mathcal{A}g(t,\,\cdot)\|_{L^p(\R^d)}\\
 &=\left\|\int_0^t\int_{\R^d}G(t-s,\,x-y)[f(s,y)^{1+\eta}-g(s,y)^{1+\eta}]\,\d y\,\d s\right\|_{L^p(\R^d)}\\
 &\lesssim \left\|\int_0^t\int_{\R^d}G(t-s,\,x-y)|f(s,y)-g(s,y)||f(s,y)^\eta+g(s,y)^\eta|\,\d y\,\d s\right\|_{L^p(\R^d)}\\
 &\lesssim \int_0^t(t-s)^{-\beta\eta d/\alpha p}\| |f(s,\cdot)-g(s,\cdot)||f(s,\cdot)^\eta+g(s,\cdot)^\eta|\|_{L^{\frac{p}{1+\eta}}(\R^d)}\,\d s\\
&\lesssim \int_0^t(t-s)^{-\beta\eta d/\alpha p}\|f(s,\cdot)-g(s,\cdot)\|_{L^p(\R^d)}[\|f(s,\cdot)\|^\eta_{L^p(\R^d)}+\| g(s,\cdot)\|_{L^p(\R^d)}^\eta] \,\d s\\
&\lesssim M^\eta \|f-g\|_{p,\theta} \int_0^t(t-s)^{-\beta\eta d/\alpha p}s^{-(1+\eta)\theta}\,\d s.
\end{align*}
Since $(1+\eta)\theta<1$ and $p> \frac{\beta d \eta}{\alpha}$, the integral in the above makes sense. We now obtain the result after some computations.
\end{proof}

\begin{proposition}\label{Lp-existence}
Let $\eta>\alpha/\beta d$ and set $q_c=\frac{\beta d\eta}{\alpha}$. Then for $\|V_0\|_{L^{q_c}(\R^d)}$ small enough, then there is a unique solution to \eqref{mild} such that 
\begin{equation*}
\|V\|_{p,\,\theta}<\infty\quad \text{for some}\quad p>q_c,
\end{equation*}
where the norm $\|\cdot\|_{p,\,\theta}$ is defined by \eqref{norm} and $\theta$ is as in Corollary \ref{contract}.
\end{proposition}
\begin{proof}
The proof is a usual fixed point argument as in say the proof of Theorem 15.2 of \cite{Souplet}. We assume that $\|V_0\|_{L^{q_c}(\R^d)}<M$ for some $M>0$.  Let 
\begin{equation*}
B_M:=\{V(t,\,x)\in L^p(\R^d); \|V\|_{p,\theta}<M \},
\end{equation*}
and 
\begin{equation*}
I(V)(t,\,x):=\mathcal{G}V_0(t,\,x)+\mathcal{A}V(t,\,x).
\end{equation*}
Then one can show that the map $I: B_M\rightarrow B_M$ has a unique fixed point whenever $M$ is small enough.
\end{proof}

\begin{proof}[Proof of Theorem \ref{thm-norm-existence}]
We choose a finitely supported initial function $V_0(x)$ which is bounded above by a small positive constant so that we can use the above result and the first part of Lemma \ref{young}. The above result says that we have a global solution satisfying 
\begin{equation*}
\|V(t,\,\cdot) \|_{L^p(\R^d)}\lesssim t^{-\theta}\quad \text{for all}\quad t>0, 
\end{equation*}
where $p>q_c$ is such that $\theta/\beta<1$. Let $p_1>p$ so that $\frac{1+\eta}{p}-\frac{1}{p_1}<\frac{\alpha}{d}$. Now from \eqref{B1},
\begin{equation*}\label{B1}
\| \int_{\R^d}G(t-s,\,\cdot-y)V(s,y)^{1+\eta}\,\d y\|_{L^{p_1}(\R^d)} \lesssim (t-s)^{-\frac{\beta d}{\alpha}(\frac{1+\eta}{p}-\frac{1}{p_1})}\|V(s,\,\cdot)\|_{L^p(\R^d)}^{1+\eta}.
\end{equation*}
This means that we have $ \|\mathcal{A}V(t,\,\cdot)\|_{L^{p_1}(\R^d)}\lesssim t^{-\tilde{\theta}_{p,\,p_1}}$ with $(1+\eta)\tilde{\theta}_{p,\,p_1}<1$.  For that particular $p_1$, we can apply the first part of Lemma \ref{young} to see that $\|\mathcal{G}V_0(t,\,\cdot) \|_{L^{p_1}(\R^d)}$ is also bounded for each $t>0$. For any $T>0$,  we can conclude that the solution $\|V(t,\,\cdot)\|_{L^{p_1}(\R^d)}$ is bounded on $(0,\,T].$ We continue the above procedure to conclude that there exists some constant $\gamma$ such  that $\|V(t,\,\cdot)\|_{L^{\infty}(\R^d)}\lesssim t^{\gamma}$ on $(0,\,T]$. Since $T$ is arbritary, this completes the proof.
\end{proof}

\section{Proof of Theorem \ref{thm-decay}}\label{proof4}
Throughout this section, we will assume that $d<\alpha.$  As seen above, the $G(t,\,x)$ does not satisfy the semigroup property. However, we can use \eqref{heat-comp} to obtain
\begin{align*}
\int_{\R^d}G(s,\,x-y)G(t,\,y-z)\,\d y&\lesssim \int_{\R^d}p(s^\beta,\,x-y)p(t^\beta,\,y-z)\,\d y\\
&\lesssim p((t+s)^\beta,\,x-z)\\
&\lesssim G(t+s,\,x-z).
\end{align*}
A straightforward consequence of the above is the following proposition where $\gamma$ will be a strictly positive constant; we will assume this throughout this section.
\begin{proposition}\label{initial}
If $V_0(x)\leq \delta G(\gamma,\,x)$, for some constant $\delta>0$, then
\begin{equation*}
\int_{\R^d}G(t,\,x-y)V_0(y)\,\d y\lesssim \delta G(t+\gamma,\,x),\quad\text{for all}\quad t>0\quad\text{and}\quad x\in \R^d.
\end{equation*}
\end{proposition}
\begin{proof}
Using the above we obtain
\begin{align*}
\int_{\R^d}G(t,\,x-y)V_0(y)\,\d y&\leq \delta\int_{\R^d}G(t,\,x-y)G(\gamma,\,y)\,\d y\\
&\lesssim  \delta G(t+\gamma,\,x).
\end{align*}
\end{proof}

\begin{proposition}\label{G2}
Suppose $\eta>\frac{\alpha}{\beta d}$, then for all $t>0$ and $x\in \R^d$,
\begin{equation*}
\int_{\R^d}\int_0^tG(t-s,\,x-y)G(s+\gamma,\,y)^{\eta+1}\,\d s\,\d y\lesssim G(t+\gamma,\,x).
\end{equation*}
\end{proposition}
\begin{proof}
We have
\begin{align*}
\int_0^t\int_{\R^d}G(t-s,\,x-y)&G(s+\gamma,\,y)^{\eta+1}\,\d s\,\d y\\
&\lesssim \int_0^t\sup_{y\in \R^d}G(s+\gamma,\,y)^\eta\int_{\R^d}G(t-s,\,x-y)G(s+\gamma,\,y)\,\d s\,\d y\\
&\lesssim G(t+\gamma,\,x)\int_0^t \frac{1}{(s+\gamma)^{\eta \beta d/\alpha}}\,\d s.
\end{align*}
Since $\gamma>0$, some calculus finishes the proof.
\end{proof}

\begin{proposition}\label{A1}
Suppose $\eta>\frac{\alpha}{\beta d}$, then
\begin{equation*}
\sup_{t>0,\,x\in\R^d}\left|\frac{(\mathcal{A}V)(t,\,x)}{G(t+\gamma,\,x)}\right|\lesssim \sup_{t>0,\,x\in\R^d}\left|\frac{V(t,\,x)}{G(t+\gamma,\,x)}\right|^{1+\eta}.
\end{equation*}
\end{proposition}

\begin{proof}
We have
\begin{align*}
\int_0^t\int_{\R^d}G(t-s,\,x-y)&V(s,y)^{1+\eta}\,\d y\,\d s\\
&\leq \int_0^t\int_{\R^d}G(t-s,\,x-y)G(s+\gamma,\,y)^{1+\eta}\left|\frac{V(s,y)}{G(s+\gamma,\,y)}\right|^{1+\eta}\,\d y\,\d s\\
&\leq \sup_{t>0,\,y\in \R^d}\left|\frac{V(t,y)}{G(t+\gamma,\,y)}\right|^{1+\eta}\int_0^t\int_{\R^d}G(t-s,\,x-y)G(s+\gamma,\,y)^{1+\eta}\,\d y\,\d s.
\end{align*}
We now use Proposition \ref{G2} to complete the proof.
\end{proof}
We need one final result before the proof of Theorem \ref{thm-decay}.
\begin{proposition}\label{contract1}
Suppose that $\eta>\frac{\alpha}{\beta d}$ and
\begin{equation*}
\sup_{t>0,\,x\in \R^d}\left|\frac{V(t,\,x)}{G(t+\gamma,\,x)}\right|\leq M\quad \text{and}\quad \sup_{t>0,\,x\in \R^d}\left|\frac{W(t,\,x)}{G(t+\gamma,\,x)}\right|\leq M,
\end{equation*}
for some $M>0$, then we have
\begin{equation*}
\sup_{t>0,\,x\in \R^d}\left|\frac{(\mathcal{A}V)(t,\,x)-(\mathcal{A}W)(t,\,x)}{G(t+\gamma,\,x)}\right|\lesssim M^\eta \sup_{t>0,\,x\in \R^d}\left|\frac{V(t,\,x)-W(t,\,x)}{G(t+\gamma,\,x)}\right|.
\end{equation*}
\begin{proof}
We start off by writing
\begin{align*}
(\mathcal{A}V)(t,\,x)&-(\mathcal{A}W)(t,\,x)\\
&=\int_0^t\int_{\R^d}G(t-s,\,x-y)[V(s,y)^{1+\eta}-W(s,y)^{1+\eta}]\,\d y\,\d s\\
&\lesssim \int_0^t\int_{\R^d}G(t-s,\,x-y)[|V(s,y)-W(s,y)|][V(s,y)^\eta+W(s,y)^\eta]\,\d y\,\d s\\
&\lesssim M^{\eta}\int_0^t\int_{\R^d}G(t-s,\,x-y)G(s+\gamma,\,y)^{1+\eta}\frac{|V(s,y)-W(s,y)|}{G(t+\gamma,\,y)}\,\d y\,\d s\\
&\lesssim M^\eta \sup_{t>0, x\in\R^d}\frac{|V(t,x)-W(t,x)|}{G(t+\gamma,\,x)}\int_0^t\int_{\R^d}G(t-s,\,x-y)G(t+\gamma,\,y)^{1+\eta}\,\d y\,\d s.
\end{align*}
An application of Proposition \ref{G2} yields the desired result.
\end{proof}
\end{proposition}
We set
\begin{equation}\label{norm}
\|V\|:=\sup_{t>0, x\in \R^d}\left|\frac{V(t,\,x)}{G(t+\gamma,\,x)}\right|.
\end{equation}
The proof of Theorem \ref{thm-decay} involves a Picard iteration which we define as follows. For $n\geq 0$,
\begin{equation}\label{iter}
V_{n+1}(t,\,x):=\int_{\R^d}G(t,\,x-y)V_0(y)\,\d y+ (\mathcal{A}V_n)(t,\,x).
\end{equation}
\begin{proof}[Proof of Theorem \ref{thm-decay}]
We have all the ingredients to follow the proof of \cite{Fujita}.  We leave it to furnish a proof.
\end{proof}
\section{Proof of Theorem \ref{thm-Dirichlet-additional-drift}}\label{proof5}
The proof of this theorem relies on the following spectral decomposition of the Dirichlet heat kernel,
\begin{equation}\label{repre}
G_D(t,\,x,\,y)=\sum_{n=1}^\infty E_\beta(-\nu_nt^\beta)\phi_n(x)\phi_n(y).
\end{equation}
$\nu_n$ are the eigenvalues  of the the fractional Laplacian on the domain $B(0,\,R)$ and the corresponding eigenfunctions $\{\phi_n\}_{n\geq 1}$ form an orthonormal basis of $L^2(B(0,\,R))$. Here $E_\beta(t)=\sum_{k=0}^\infty t^{\beta k}/\Gamma(1+\beta k)$ is the Mittag-Leffler function.  See \cite{CME} and \cite{mnv-09} for more information about this.  If $\beta$ were one, then the above representation would have been in terms of the exponential function instead of the Mittag-Leffler function. The key observation is that we have the following polynomial decay:
\begin{equation}\label{decay}
\frac{1}{1+\Gamma(1-\beta)t}\leq E_\beta(-t)\leq \frac{1}{1+\Gamma(1+\beta)^{-1}t}\quad \text{for all}\quad t>0.
\end{equation}
We will only need the lower bound for the proof. The proof follows the same idea as that of Kaplan \cite{Kaplan}.
\begin{proof}[Proof of Theorem \ref{thm-Dirichlet-additional-drift}.]
Set
\begin{equation*}
F(t):=\int_{B(0,\,R)}V(t,\,x)\phi_1(x)\,\d x.
\end{equation*}
We now use the integral formulation of the equation given by \eqref{mild2} together with the representation \eqref{repre} to write
\begin{align*}
F(t)&=E_\beta(-\mu_1t^\beta)\int_{B(0,\,R)}V(t,\,y)\phi_1(y)\,\d y+\int_0^tE_\beta(-\mu_1(t-s)^\beta)\int_{B(0,\,R)}\phi_1(y)V(s,\,y)^{1+\eta}\,\d y\,\d s\\
&\gtrsim E_\beta(-\mu_1t^\beta)K_{V_0, \phi_1}+\int_0^tE_\beta(-\mu_1(t-s)^\beta)F(s)^{1+\eta}\d s\\
&\gtrsim \frac{K_{V_0, \phi_1}}{t^\beta}+\int_0^t\frac{F(s)^{1+\eta}}{t^\beta}\d s,
\end{align*}
where we have also taken $t$ to be large enough. We now let $G(t):=t^\beta F(t)$ and consider the case $\beta(1+\eta)<1$. Then the above inequality reduces to
\begin{align*}
G(t)\gtrsim K_{V_0, \phi_1}+\int_0^t\frac{G(s)^{1+\eta}}{s^{\beta(1+\eta)}}\,\d s.
\end{align*}
Since $G(t)$ is a supersolution to the following non-linear ordinary differential equation.

\begin{equation*}
\frac{\tilde{G}'(s)}{\tilde{G}(s)^{1+\eta}}=\frac{1}{s^{\beta(1+\eta)}}\quad\text{with}\quad \tilde{G}(0)=K_{V_0, \phi_1}.
\end{equation*}
Therefore there exists a $t_0$ such that $G(t)=\infty$ for all $t\geq t_0$ no matter what the initial condition $K_{V_0, \phi_1}$ is. When $\beta(1+\eta)\geq1$ and  $K_{V_0, \phi_1}>0$, we now obtain 
\begin{align*}
G(t)\gtrsim K_{V_0, \phi_1}+\int_1^t\frac{G(s)^{1+\eta}}{s^{\beta(1+\eta)}}\,\d s
\end{align*}
which can now be compared with
\begin{equation*}
\frac{\tilde{G}'(s)}{\tilde{G}(s)^{1+\eta}}=\frac{1}{s^{\beta(1+\eta)}}\quad\text{with}\quad \tilde{G}(1)=K_{V_0, \phi_1}.
\end{equation*}
Therefore there exists a $t_1$ such that $G(t)=\infty$ for all $t\geq t_1$ provided that the initial condition $K_{V_0, \phi_1}$ is large enough. This finishes the proof since $\phi_1(x)$ is strictly positive.
\end{proof}
\bibliography{Foon-Nane}

\end{document}